\documentclass[11pt]{amsart}
\usepackage{amssymb, amsthm, amsmath, gensymb, dsfont, cite}
\usepackage{graphicx, comment}
\usepackage[small]{caption}
\usepackage{subcaption}
\usepackage{epsfig}
\usepackage{tikz, pgfplots, float}
\usepackage{setspace}
\usepackage{amsfonts}

\usepackage[utf8]{inputenc}
\usepackage{fullpage}
\usepackage{framed}

\usepackage{enumerate}
\usepackage{url}
\usepackage[breaklinks]{hyperref}
\usepackage{cleveref}
\hypersetup{
	colorlinks = true, % colors links instead of ugly boxes
	urlcolor = cyan, % color for external hyperlinks
	linkcolor = teal, % color of internal links
	citecolor = cyan % color of citations
}

\newtheorem{theorem}{Theorem}

\newtheorem{lemma}[theorem]{Lemma}

\crefname{claim}{claim}{claims}
\newtheorem{observation}[theorem]{Observation}
\newtheorem{proposition}[theorem]{Proposition}
\newtheorem{corollary}[theorem]{Corollary}

\title{On the Ramsey classes of random hypergraphs}

\author{Dingyuan Liu}

\address{Karlsruhe Institute of Technology, Englerstraße 2, D-76131 Karlsruhe, Germany}

\email{liu@mathe.berlin}

\begin{document}
\maketitle

\vspace{-1.7em}
\begin{abstract}
Let $r,s,t\geq2$ be integers. For $r$-graphs $G$ and $F_1,\dots,F_s$, we write $G\to(F_1,\dots,F_s)$ if every $s$-edge-coloring of $G$ yields a monochromatic copy of $F_i$ in the $i$-th color for some $1\leq i\leq s$. Let $\mathcal{R}(F_1,\dots,F_s)$ denote the family of all $r$-graphs $G$ with $G\to(F_1,\dots,F_s)$. When $F_1=\dots=F_s=F$, we write $\mathcal{R}(F;s)=\mathcal{R}(F_1,\dots,F_s)$.

In this paper, we investigate when $\mathcal{R}(H;s)\subseteq\mathcal{R}(Q_1,\dots,Q_t)$ holds, where $H=H^{(r)}(n,p)$ is a random $r$-graph and $Q_1,\dots,Q_t$ are fixed $r$-graphs. Our main result determines the threshold for a large class of such $Q_1,\dots,Q_t$, including complete $r$-graphs. The key ingredient in our proof is a generalization of a result of Graham, {\L}uczak, R\"odl, and Ruci\'nski, which provides a necessary and sufficient condition for $\mathcal{R}(F_1,\dots,F_s)\subseteq\mathcal{R}(Q_1,\dots,Q_t)$, where $Q_1,\dots,Q_t$ are highly connected. As a byproduct, we characterize when two tuples of highly connected $r$-graphs are Ramsey equivalent.
\end{abstract}

\section{Introduction}
\label{introduction}
\subsection{Background}
Ramsey theory is a central area of mathematics concerned with the unavoidable emergence of order in large discrete structures, typically modeled by hypergraphs. Let $r\geq2$ be an integer. An $r$-uniform hypergraph (abbreviated as an $r$-graph) $F$ is an ordered pair $(V(F),E(F))$, where $V(F)$ is the vertex set and $E(F)\subseteq\binom{V(F)}{r}$ is the edge set. When $r=2$, an $r$-graph is simply referred to as a graph. Given $r$-graphs $G$ and $F_1,\dots,F_s$, we say that $G$ is Ramsey for $(F_1,\dots,F_s)$, denoted by $G\to(F_1,\dots,F_s)$, if every $s$-edge-coloring of $G$ yields a monochromatic copy of $F_i$ in the $i$-th color for some $1\leq i\leq s$. The Ramsey number $R(F_1,\dots,F_s)$ is defined as the minimum number of vertices in an $r$-graph $G$ satisfying $G\to(F_1,\dots,F_s)$. The existence of such numbers was first established by Ramsey~\cite{ramsey1930}, and have since given rise to a rich body of research. For further background and recent developments concerning graph and hypergraph Ramsey numbers, we refer the reader to the surveys of Morris~\cite{morris2026} and Mubayi and Suk~\cite{mubayi2020}.

A prominent line of research in Ramsey theory focuses on Ramsey properties of random graphs and hypergraphs. Denote by $v(F)$ and $e(F)$ the numbers of vertices and edges of $F$, respectively. The maximum $r$-density of an $r$-graph $F$ is defined as
\[m_r(F)=\begin{cases}
0 & \text{if $e(F)=0$,}\\
1/r & \text{if $e(F)=1$,}\\
\max_{J\subseteq F,v(J)>r}\frac{e(J)-1}{v(J)-r} & \text{if $e(F)\geq2$.}
\end{cases}\]
For $n\in\mathbb{N}$ and $0\leq p\leq 1$, let $G(n,p)$ be a random graph on the vertex set $[n]$, where each pair of vertices appears as an edge independently with probability $p$. R\"odl and Ruci\'nski~\cite{rodl1995} showed that for any integer $s\geq2$ and any graph $F$ containing a cycle, there are constants $C,D>0$ such that
\begin{equation}
\label{random_Ramsey}
\lim_{n\to\infty}\mathbb{P}\big(G(n,p)\to(\underbrace{F,\dots,F}_{\text{$s$ times}})\big)=\begin{cases}
1 & \text{if $p\geq Cn^{-1/m_2(F)}$,} \quad\quad\quad\text{($1$-statement)}\\
0 & \text{if $p\leq Dn^{-1/m_2(F)}$.} \quad\quad\quad\text{($0$-statement)}
\end{cases}
\end{equation}
See also~\cite{nenadov2016} for a short proof of the above result.

The $1$-statement in~\eqref{random_Ramsey} was later extended to hypergraphs by Friedgut, R\"odl, and Schacht~\cite{friedgut2010}, and independently by Conlon and Gowers~\cite{conlon2016}. Let $H^{(r)}(n,p)$ be a random $r$-graph on the vertex set $[n]$, where every $r$-subset of $[n]$ appears as an edge independently with probability $p$. The above groups of authors showed that for any $s\geq2$ and any $r$-graph $F$ containing at least one edge, there is a constant $C>0$ such that
\[\lim_{n\to\infty}\mathbb{P}\big(H^{(r)}(n,p)\to(\underbrace{F,\dots,F}_{\text{$s$ times}})\big)=1\quad\text{if $p\geq Cn^{-1/m_r(F)}$}.\]
It was noted in~\cite{conlon2016,friedgut2010,rodl1998} that the corresponding $0$-statement should also hold for ``most'' $r$-graphs $F$, but this problem remains open except for several special cases (see~\cite{gugelmann2017,nenadov2017,thomas2013} and the references therein). Notably, Nenadov, Person,~\v{S}kori\'c, and Steger~\cite{nenadov2017} established the $0$-statement for $F$ being complete $r$-graphs.

There is also a substantial body of work on the asymmetric Ramsey properties of random graphs and hypergraphs~\cite{bowtell2023,christoph2025,gugelmann2017,kohayakawa1997,kohayakawa2014,kuperwasser2025,liebenau2023,marciniszyn2009,mousset2020}. Naturally, the parameter governing the threshold is an asymmetric analogue of the maximum $r$-density. Given two $r$-graphs $T$ and $F$ with $m_r(T)\geq m_r(F)>0$, the asymmetric maximum $r$-density of the ordered pair $(T,F)$ is defined as
\[m_r(T,F)=\max_{J\subseteq T,v(J)\geq r}\frac{e(J)}{v(J)-r+m_r(F)^{-1}}.\]
It is known (see, e.g.,~\cite[Proposition~3.1]{bowtell2023}) that $m_r(T)\geq m_r(T,F)\geq m_r(F)$, and if $m_r(T)>m_r(F)$ then $m_r(T)>m_r(T,F)>m_r(F)$. In particular, when $T=F$, we have $m_r(T,F)=m_r(F)$.

Resolving the well-known Kohayakawa--Kreuter conjecture~\cite{kohayakawa1997}, Christoph, Martinsson, Steiner, and Wigderson~\cite{christoph2025}, building on earlier work~\cite{bowtell2023,kuperwasser2025,mousset2020}, proved that for any fixed graphs $F_1,\dots,F_s$ satisfying $m_2(F_1)\geq\dots\geq m_2(F_s)$ and $m_2(F_2)>1$, there are constants $C,D>0$ such that
\begin{equation}
\label{asymmetric_random_Ramsey}
\lim_{n\to\infty}\mathbb{P}\big(G(n,p)\to(F_1,\dots,F_s)\big)=\begin{cases}
1 & \text{if $p\geq Cn^{-1/m_2(F_1,F_2)}$,}\\
0 & \text{if $p\leq Dn^{-1/m_2(F_1,F_2)}$.}
\end{cases}
\end{equation}
For random hypergraphs, a $1$-statement analogous to that in~\eqref{asymmetric_random_Ramsey} was proved by Mousset, Nenadov and Samotij~\cite{mousset2020}, while the corresponding $0$-statement was recently established by Bowtell, Hancock, and Hyde~\cite{bowtell2023} for $r$-graphs under an additional condition. We shall present their results in~\Cref{preliminaries}.

\subsection{Our results}
In this paper, we investigate the Ramsey classes of random hypergraphs. This problem differs from, but is closely related to, those discussed above. Given $r$-graphs $F_1,\dots,F_s$, the Ramsey class of $(F_1,\dots,F_s)$, denoted by $\mathcal{R}(F_1,\dots,F_s)$, is the family of all $r$-graphs $G$ with $G\to(F_1,\dots,F_s)$. When $F_1=\dots=F_s=F$, we simply write $\mathcal{R}(F;s)=\mathcal{R}(F_1,\dots,F_s)$.

Let $H=H^{(r)}(n,p)$ be a random $r$-graph. It is immediate that for any integers $s,t\geq2$ and fixed $r$-graphs $Q_1,\dots,Q_t$, we have $\mathcal{R}(Q_1,\dots,Q_t)\not\subseteq\mathcal{R}(H;s)$ whenever $n>R(Q_1,\dots,Q_t)$. Indeed, by the definition of $R(Q_1,\dots,Q_t)$, there exists an $r$-graph $G$ with $v(G)=R(Q_1,\dots,Q_t)$ such that $G\in\mathcal{R}(Q_1,\dots,Q_t)$. On the other hand, since $v(G)<v(H)$ we have $G\not\in\mathcal{R}(H;s)$.

Therefore, the interesting and non-trivial question is under what conditions the reverse inclusion $\mathcal{R}(H;s)\subseteq\mathcal{R}(Q_1,\dots,Q_t)$ holds. Let $m(F)=\max_{J\subseteq F}e(J)/v(J)$ denote the maximum density of an $r$-graph $F$. A generalization of a result of Bollob\'as~\cite{bollobas1981} (see also a proof in~\cite{thomas2013}) asserts that
\begin{equation}
\label{small_subgraphs}
\lim_{n\to\infty}\mathbb{P}\big(\text{$H^{(r)}(n,p)$ contains a copy of $F$}\big)=\begin{cases}
        1 & \text{if $p\gg n^{-1/m(F)}$},\\
        0 & \text{if $p\ll n^{-1/m(F)}$}.
    \end{cases}
\end{equation}
Using~\eqref{small_subgraphs} we have the following observation for the case $s\geq t$.

\begin{observation}
\label{slarger}
Let $r,s,t\geq2$ be integers with $s\geq t$. Let $Q_1,\dots,Q_t$ be $r$-graphs and suppose that $m(Q_1)\geq\dots\geq m(Q_t)$. Let $H=H^{(r)}(n,p)$. Then
\[\lim_{n\to\infty}\mathbb{P}\big(\mathcal{R}(H;s)\subseteq\mathcal{R}(Q_1,\dots,Q_t)\big)=\begin{cases}
        1 & \text{if $p\gg n^{-1/m(Q_1)}$},\\
        0 & \text{if $p\ll n^{-1/m(Q_1)}$ and $Q_1$ is Ramsey-avoidable}.
    \end{cases}\]
Here, $Q_1$ is called Ramsey-avoidable if, for every $Q_1$-free $r$-graph $F$, there exists a $Q_1$-free $r$-graph in $\mathcal{R}(F;2)$.
\end{observation}

Indeed, since $s\geq t$, $\mathcal{R}(H;s)\subseteq\mathcal{R}(H;t)$. So for the $1$-statement in~\Cref{slarger}, it suffices to show that $\mathcal{R}(H;t)\subseteq\mathcal{R}(Q_1,\dots,Q_t)$, which holds if $H$ contains a copy of $Q_j$ for every $j\in[t]$. As $m(Q_1)\geq\dots\geq m(Q_t)$ and $p\gg n^{-1/m(Q_1)}$, by~\eqref{small_subgraphs}, with high probability $H$ contains a copy of $Q_j$ for each $j\in[t]$. Consequently, $\mathcal{R}(H;s)\subseteq\mathcal{R}(Q_1,\dots,Q_t)$ holds with high probability. The $0$-statement follows since $H$ is with high probability $Q_1$-free and $Q_1$ is Ramsey-avoidable.

It is known that all complete bipartite graphs except $K_{2,2}$~\cite{nesetril1992}, and all complete $r$-graphs~\cite{nesetril1979} are Ramsey-avoidable. Moreover, in~\cite{nesetril1977,nesetril1989} Ne\v{s}et\v{r}il and R\"odl introduced two large families $\mathcal{X}_r$ and $\mathcal{Y}_r$ of Ramsey-avoidable $r$-graphs, which we will describe later.

We now turn to the case $s<t$, where the situation becomes more delicate. As in other Ramsey problems on random hypergraphs mentioned earlier, the $1$-statement is typically more tractable. Therefore, we begin by establishing a general $1$-statement in this regime.

\begin{proposition}
\label{our_1_statement}
Let $r,s,t\geq2$ be integers with $s<t$. Let $Q_1,\dots,Q_t$ be $r$-graphs such that $m_r(Q_1)\geq\dots\geq m_r(Q_t)>0$ and $m_r(Q_{s+1})>1$. Let $\delta=\max\{\mu,\sigma\}$, where $\mu=\max\{m(Q_1),\dots,m(Q_t)\}$ and $\sigma=\min\{m_r(Q_i,Q_{s+1}):i\in[s]\}$. Then
\[\lim_{n\to\infty}\mathbb{P}\big(\mathcal{R}(H^{(r)}(n,p);s)\subseteq\mathcal{R}(Q_1,\dots,Q_t)\big)=1\quad\text{if $p\gg n^{-1/\delta}$}.\]
\end{proposition}

We note that the condition $m_r(Q_{s+1})>1$ in~\Cref{our_1_statement} can be removed by a minor refinement of the argument. Nevertheless, as $m_r(Q_{s+1})>1$ is required in the subsequent $0$-statement, we retain it here for consistency.

The proof of~\Cref{our_1_statement} is relatively straightforward and builds on the following observation. For two tuples $(F_1,\dots,F_s)$ and $(Q_s,\dots,Q_t)$ of $r$-graphs, if there exists a partition $A_1\dot\cup\dots\dot\cup A_s$ of $[t]$ such that $F_i\to(Q_j)_{j\in A_i}$ for all $i\in[s]$, then $\mathcal{R}(F_1,\dots,F_s)\subseteq\mathcal{R}(Q_1,\dots,Q_t)$. This gives a sufficient condition for $\mathcal{R}(F_1,\dots,F_s)\subseteq\mathcal{R}(Q_1,\dots,Q_t)$, and, together with asymmetric Ramsey properties of random hypergraphs, allows us to establish the $1$-statement.

Naturally, one of the main obstacles in proving a $0$-statement lies in identifying an appropriate necessary condition for $\mathcal{R}(F_1,\dots,F_s)\subseteq\mathcal{R}(Q_1,\dots,Q_t)$. A classical result of Graham, {\L}uczak, R\"odl, and Ruci\'nski~\cite{graham2002} shows that, when $F_1,\dots,F_s$ and $Q_1,\dots,Q_t$ are all graph cliques, the sufficient condition above is also necessary. Here, we extend this result to a more general setting, by lifting the restriction on $F_1,\dots,F_s$ and allowing $Q_1,\dots,Q_t$ to range over a large class of $r$-graphs.

We say that an $r$-graph $F$ is non-trivially connected if, for any vertex cut $S\subseteq V(F)$, the induced $r$-graph $F[S]$ is not $r$-partite. Let $\mathcal{X}_r$ be the family of all non-trivially connected $r$-graphs. Let $\mathcal{Y}_r$ be the family of all $r$-graphs, in which every pair of vertices is contained in an edge. We note that although $\mathcal{X}_r$ and $\mathcal{Y}_r$ have a large intersection, neither family is contained in the other. For example, an $r$-graph obtained from a large complete $r$-graph by deleting all edges containing two fixed vertices belongs to $\mathcal{X}_r$ but not to $\mathcal{Y}_r$. On the other hand, the family $\mathcal{Y}_r$ contains all complete $r$-graphs, whereas $\mathcal{X}_r$ contains only complete $r$-graphs on at least $2r$ vertices.

\begin{theorem}
\label{structural_theorem}
Let $r\geq2$ and $s,t\geq1$ be integers. Let $F_1,\dots,F_s$ and $Q_1,\dots,Q_t$ be $r$-graphs, where $Q_1,\dots,Q_t\in\mathcal{X}_r\cup\mathcal{Y}_r$. Then
\[\mathcal{R}(F_1,\dots,F_s)\subseteq\mathcal{R}(Q_1,\dots,Q_t)\]
if and only if there exists a partition $A_1\dot\cup\dots\dot\cup A_s$ of $[t]$ such that $F_i\to(Q_j)_{j\in A_i}$ for all $i\in[s]$.
\end{theorem}

Applying~\Cref{structural_theorem} we are able to establish the $0$-statement for $\mathcal{R}(H^{(r)}(n,p);s)\subseteq\mathcal{R}(Q_1,\dots,Q_t)$ when $Q_1,\dots,Q_s$ are chosen from $\mathcal{X}_r\cup\mathcal{Y}_r$. In the case $r\geq3$, however, an additional property is required. 

An $r$-graph $F$ is called strictly $r$-balanced if $m_r(F)>m_r(J)$ holds for all $J\subsetneq F$. For example, any complete $r$-graph on at least $r$ vertices is strictly $r$-balanced. For two $r$-graphs $T$ and $F$ with $m_r(T)\geq m_r(F)>0$, we say that $T$ is strictly $F$-balanced if $m_r(T,F)$ is attained only by $J=T$. Furthermore, following the terminology in~\cite{bowtell2023}, an ordered pair $(T',F')$ is called a heart of $(T,F)$ if
\begin{itemize}
\item $F'\subseteq F$ is strictly $r$-balanced and $m_r(F')=m_r(F)$;
\item if $m_r(T)=m_r(F)$, then $T'\subseteq T$ is strictly $r$-balanced and $m_r(T')=m_r(T)$;
\item if $m_r(T)>m_r(F)$, then $T'\subseteq T$ is strictly $F'$-balanced and $m_r(T',F')=m_r(T,F)$.
\end{itemize}
We say that an ordered pair $(T,F)$ of $r$-graphs is Ramsey-dense if there is a heart $(T',F')$ of $(T,F)$, such that every $G\in\mathcal{R}(T',F')$ satisfies $m(G)>m_r(T',F')$. With the necessary definitions in place, we state our $0$-statement as follows.

\begin{theorem}
\label{our_0_statement}
Let $r,s,t\geq2$ be integers with $s<t$. Let $Q_1,\dots,Q_t$ be $r$-graphs chosen from $\mathcal{X}_r\cup\mathcal{Y}_r$ such that $m_r(Q_1)\geq\dots\geq m_r(Q_t)>0$ and $m_r(Q_{s+1})>1$. Let $\delta=\max\{\mu,\sigma\}$, where
\[\mu=\max\{m(Q_1),\dots,m(Q_t)\}\quad\text{and}\quad\sigma=\min\{m_r(Q_i,Q_{s+1}):i\in[s]\}.\]
If $r\geq3$ and $\mu<\sigma$, assume additionally that $(Q_i,Q_{j})$ is Ramsey-dense for all $i<j\in[s+1]$. Then
\[\lim_{n\to\infty}\mathbb{P}\big(\mathcal{R}(H^{(r)}(n,p);s)\subseteq\mathcal{R}(Q_1,\dots,Q_t)\big)=0\quad\text{if $p\ll n^{-1/\delta}$}.\]
\end{theorem}

Let us remark that the seemingly redundant assumption of Ramsey-denseness is, to some extent, necessary. Indeed, when $\mu<\sigma$, the threshold for $\mathcal{R}(H^{(r)}(n,p);s)\not\subseteq\mathcal{R}(Q_1,\dots,Q_t)$ will be governed by that for $H^{(r)}(n,p)\not\to(Q_i,Q_j)$, for some $i<j\in[s+1]$. It was shown in~\cite{gugelmann2017,thomas2013} (see also~\cite{bowtell2023} for a discussion) that if $(Q_i,Q_j)$ is not Ramsey-dense, then the threshold for $H^{(r)}(n,p)\not\to(Q_i,Q_j)$ could exhibit a different exponent unrelated to $m_r(Q_i,Q_j)$. This shows that $n^{-1/\delta}$ need not be the correct threshold without the Ramsey-denseness assumption.

Bowtell, Hancock, and Hyde established in~\cite{bowtell2023} several sufficient conditions for a pair of $r$-graphs to be Ramsey-dense, one of which asserts that $(Q_i,Q_j)$ is Ramsey-dense if $Q_j$ contains a strictly $r$-balanced and highly chromatic subhypergraph (see~\Cref{Ramsey_dense}). This, together with~\Cref{our_1_statement} and~\Cref{our_0_statement}, gives the following corollary.

\begin{corollary}
\label{complete_threshold}
Let $r,s,t\geq2$ be integers with $s<t$. Let $Q_1,\dots,Q_t$ be complete $r$-graphs such that $v(Q_1)\geq\dots\geq v(Q_t)\geq r$ and $v(Q_{s+1})>r^2-r$. Then
\[\lim_{n\to\infty}\mathbb{P}\big(\mathcal{R}(H^{(r)}(n,p);s)\subseteq\mathcal{R}(Q_1,\dots,Q_t)\big)=\begin{cases}
        1 & \text{if $p\gg n^{-1/\delta}$},\\
        0 & \text{if $p\ll n^{-1/\delta}$},
    \end{cases}\]
where $\delta=\max\{m(Q_1),m_r(Q_s,Q_{s+1})\}$.
\end{corollary}

In a different direction,~\Cref{structural_theorem} is closely related to the problem of Ramsey equivalence. Two tuples of $r$-graphs are said to be Ramsey equivalent if their Ramsey classes are identical. While the systematic study of Ramsey equivalence was initiated by Szab\'o, Zumstein, and Z\"urcher~\cite{szabo2010} and has since motivated extensive research~\cite{axenovich2017,bloom2018,boyadzhiyska2024,boyadzhiyska2024',clemens2020,fox2014,savery2022}, the notion itself had appeared much earlier. In particular, Graham, {\L}uczak, R\"odl, and Ruci\'nski~\cite{graham2002} established that two tuples of complete graphs $(K_{k_1},\dots,K_{k_s})$ and $(K_{\ell_1},\dots,K_{\ell_t})$ are Ramsey equivalent if and only if the multisets $\{k_1,\dots,k_s\}$ and $\{\ell_1,\dots,\ell_t\}$ coincide. See also~\cite{boyadzhiyska2024'} for a different proof of this result. As an application of~\Cref{structural_theorem}, we extend this result to $r$-graphs chosen from $\mathcal{X}_r\cup\mathcal{Y}_r$.

\begin{theorem}
\label{Ramsey_equivalent}
Let $r\geq2$ and $s,t\geq1$ be integers. Let $F_1,\dots,F_s$ and $Q_1,\dots,Q_t$ be $r$-graphs chosen from $\mathcal{X}_r\cup\mathcal{Y}_r$, each containing at least two edges. Then
\[\mathcal{R}(F_1,\dots,F_s)=\mathcal{R}(Q_1,\dots,Q_t)\]
if and only if $(F_1,\dots,F_s)$ and $(Q_1,\dots,Q_t)$ are identical up to reordering and isomorphism.
\end{theorem}

\subsection{Organization of the paper}
The paper is organized as follows. In~\Cref{preliminaries}, we collect more results on Ramsey properties of random hypergraphs and prove~\Cref{our_1_statement}. In~\Cref{random_proof}, we prove~\Cref{our_0_statement} and~\Cref{complete_threshold}, assuming~\Cref{structural_theorem}. In~\Cref{proof_equivalent}, we prove~\Cref{Ramsey_equivalent}, again assuming~\Cref{structural_theorem}. We establish~\Cref{structural_theorem} in~\Cref{proof_structural} and~\Cref{proof_lemma}. In~\Cref{concluding} we discuss further consequences and the limitations of our method.

\section{Preliminaries and the proof of~\Cref{our_1_statement}}
\label{preliminaries}
\subsection{Preliminary results}
In this subsection, we state several results from~\cite{bowtell2023,christoph2025,mousset2020} that will be used in subsequent proofs.

The first is a $1$-statement for the asymmetric Ramsey properties of random hypergraphs due to Mousset, Nenadov and Samotij~\cite{mousset2020}.

\begin{theorem}[Mousset--Nenadov--Samotij~\cite{mousset2020}]
\label{1_statement}
Let $r,\ell\geq2$ be integers. Let $F_1,\dots,F_\ell$ be $r$-graphs such that $m_r(F_1)\geq\dots\geq m_r(F_\ell)>0$ and $m_r(F_2)>1$. Then there is a constant $C>0$ such that
\[\lim_{n\to\infty}\mathbb{P}\big(H^{(r)}(n,p)\to(F_1,\dots,F_\ell)\big)=1\quad\text{if $p\geq Cn^{-1/m_r(F_1,F_2)}$}.\]
\end{theorem}

The second is a corresponding $0$-statement established by Bowtell, Hancock, and Hyde~\cite{bowtell2023}, under an additional Ramsey-denseness assumption.

\begin{theorem}[Bowtell--Hancock--Hyde~\cite{bowtell2023}]
\label{0_statement}
Let $r\geq2$ be an integer. Let $T$ and $F$ be $r$-graphs such that $m_r(T)\geq m_r(F)>1$ and $(T,F)$ is Ramsey-dense. Then there is a constant $D>0$ such that
\[\lim_{n\to\infty}\mathbb{P}\big(H^{(r)}(n,p)\to(T,F)\big)=0\quad\text{if $p\leq Dn^{-1/m_r(T,F)}$}.\]
\end{theorem}

Furthermore, we need two results on Ramsey-denseness.

\begin{theorem}[Christoph--Martinsson--Steiner--Wigderson~\cite{christoph2025}]
\label{graph_Ramsey_dense}
For any graphs $T$ and $F$ with $m_r(T)\geq m_r(F)>1$, $(T,F)$ is Ramsey-dense.
\end{theorem}

The next is a sufficient condition for Ramsey-denseness of pairs of $r$-graphs. Given an $r$-graph $F$, the chromatic number $\chi(F)$ is the minimum number of colors needed to color $V(F)$ without a monochromatic edge.

\begin{lemma}[Bowtell--Hancock--Hyde~\cite{bowtell2023}]
\label{Ramsey_dense}
Let $r\geq2$ be an integer. Let $T$ and $F$ be $r$-graphs with $m_r(T)\geq m_r(F)>1$. If there exists a strictly $r$-balanced $F'\subseteq F$ such that $m_r(F')=m_r(F)$ and $\chi(F')>r$, then $(T,F)$ is Ramsey-dense.
\end{lemma}

\subsection{Proof of~\Cref{our_1_statement}}
Let $H=H^{(r)}(n,p)$ and $p\gg n^{-1/\delta}$. To prove that
\[\mathcal{R}(H;s)\subseteq\mathcal{R}(Q_1,\dots,Q_t),\]
it suffices to give a partition $A_1\dot\cup\dots\dot\cup A_s$ of $[t]$ such that $H\to(Q_j)_{j\in A_i}$ for all $i\in[s]$.

Let $x\in[s]$ such that $\sigma=m_r(Q_x,Q_{s+1})$. We consider the partition where $A_i=\{i\}$ for $i\in[s]\setminus\{x\}$ and $A_x=\{x,s+1,\dots,t\}$. Since $\delta\geq\mu=\max\{m(Q_1),\dots,m(Q_t)\}$, by~\eqref{small_subgraphs} with high probability $H$ contains a copy of $Q_j$ for each $j\in[t]$. Thus, $H\to(Q_j)_{j\in A_i}$ for $i\in[s]\setminus\{x\}$ with high probability. Moreover, since $m_r(Q_x)\geq m_r(Q_{s+1})\geq\dots\geq m_r(Q_t)$, $m_r(Q_{s+1})>1$, and
\[p\gg n^{-1/\delta}\geq n^{-1/\sigma}=n^{-1/m_r(Q_x,Q_{s+1})},\]
by~\Cref{1_statement}, $H\to(Q_j)_{j\in A_x}=(Q_x,Q_{s+1},\dots,Q_t)$ almost surely. Therefore, with high probability we have $\mathcal{R}(H;s)\subseteq\mathcal{R}(Q_1,\dots,Q_t)$.\qed

\section{Proofs of~\Cref{our_0_statement} and~\Cref{complete_threshold}}
\label{random_proof}
\subsection{Proof of~\Cref{our_0_statement}}
Recall that $Q_1,\dots,Q_t$ are $r$-graphs chosen from $\mathcal{X}_r\cup\mathcal{Y}_r$ such that $m_r(Q_1)\geq\dots\geq m_r(Q_t)>0$ and $m_r(Q_{s+1})>1$. Let $\delta=\max\{\mu,\sigma\}$, where
\[\mu=\max\{m(Q_1),\dots,m(Q_t)\}\quad\text{and}\quad\sigma=\min\{m_r(Q_i,Q_{s+1}):i\in[s]\}.\]

Let $H=H^{(r)}(n,p)$ and $p\ll n^{-1/\delta}$. If $\delta=\mu$, then~\eqref{small_subgraphs} implies that, with high probability, $H$ is $Q_j$-free for some $j\in[t]$. A result of Ne\v{s}et\v{r}il and R\"odl~\cite{nesetril1989} (see also~\Cref{partition_Ramsey}) shows that if an $r$-graph $H$ is $Q_j$-free, where $Q_j\in\mathcal{X}_r\cup\mathcal{Y}_r$, then there exists an $r$-graph $G\in\mathcal{R}(H;s)$ that is also $Q_j$-free. Since $G$ is $Q_j$-free, we have that $G\not\in\mathcal{R}(Q_1,\dots,Q_t)$. Thus, $\mathcal{R}(H;s)\not\subseteq\mathcal{R}(Q_1,\dots,Q_t)$ holds with high probability.

If $\delta\neq\mu$, then $\delta=\sigma>\mu$. In this case, by~\Cref{graph_Ramsey_dense} and the assumption for $r\geq3$, we have that $(Q_i,Q_{j})$ is Ramsey-dense for all $i<j\in[s+1]$. Let $x\in[s]$ such that $\delta=\sigma=m_r(Q_x,Q_{s+1})$. By~\Cref{structural_theorem}, to establish
\[\mathcal{R}(H;s)\not\subseteq\mathcal{R}(Q_1,\dots,Q_t),\]
it suffices to prove that, for any partition $A_1\dot\cup\dots\dot\cup A_s$ of $[t]$, there exists some $i\in[s]$ such that $H\not\to(Q_j)_{j\in A_i}$. Since the number of such partitions is bounded by a constant, it suffices to show that for any fixed partition $A_1\dot\cup\dots\dot\cup A_s$ of $[t]$, with high probability $H\not\to(Q_j)_{j\in A_i}$ for some $i\in[s]$. Let $A\in\{A_1,\dots,A_s\}$ be the part that contains $x$. If $A$ contains some $a$ with $1\leq a<x$, then, since $m_r(Q_a,Q_x)\geq m_r(Q_x)\geq m_r(Q_x,Q_{s+1})$, $(Q_a,Q_x)$ is Ramsey-dense, and
\[p\ll n^{-1/\delta}=n^{-1/m_r(Q_x,Q_{s+1})}\leq n^{-1/m_r(Q_a,Q_x)},\]
by~\Cref{0_statement}, $H\not\to(Q_a,Q_x)$ and thus $H\not\to(Q_j)_{j\in A}$ holds with high probability.

If $A$ contains some $b$ with $x<b\leq s+1$, then, since $m_r(Q_x,Q_b)\geq m_r(Q_x,Q_{s+1})$, $(Q_x,Q_b)$ is Ramsey-dense, and
\[p\ll n^{-1/\delta}=n^{-1/m_r(Q_x,Q_{s+1})}\leq n^{-1/m_r(Q_x,Q_b)},\]
by~\Cref{0_statement}, $H\not\to(Q_x,Q_b)$ and thus $H\not\to(Q_j)_{j\in A}$ holds with high probability.

If $A\cap\left([s+1]\setminus\{x\}\right)=\emptyset$, then by the pigeonhole principle there exists an $A'\in\{A_1,\dots,A_s\}\setminus\{A\}$ which contains some $a<b\in[s+1]\setminus\{x\}$. Then since $m_r(Q_a,Q_b)\geq m_r(Q_a,Q_{s+1})\geq m_r(Q_x,Q_{s+1})$, $(Q_a,Q_b)$ is Ramsey-dense, and
\[p\ll n^{-1/\delta}=n^{-1/m_r(Q_x,Q_{s+1})}\leq n^{-1/m_r(Q_a,Q_b)},\]
by~\Cref{0_statement}, $H\not\to(Q_a,Q_b)$ and thus $H\not\to(Q_j)_{j\in A'}$ holds with high probability. This completes the proof.\qed

\subsection{Proof of~\Cref{complete_threshold}}
Let $Q_1,\dots,Q_t$ be complete $r$-graphs with $v(Q_1)\geq\dots\geq v(Q_t)\geq r$ and $v(Q_{s+1})>r^2-r$. It follows that $m(Q_1)\geq\dots\geq m(Q_t)$, $m_r(Q_1)\geq\dots\geq m_r(Q_t)>0$, and $m_r(Q_{s+1})>1$. Moreover, as $Q_i$ contains a copy of $Q_j$ for all $i\leq j$, we have that $m_r(Q_s,Q_{s+1})=\min\{m_r(Q_i,Q_{s+1}):i\in[s]\}$. Then by~\Cref{our_1_statement},
\[\lim_{n\to\infty}\mathbb{P}\big(\mathcal{R}(H^{(r)}(n,p);s)\subseteq\mathcal{R}(Q_1,\dots,Q_t)\big)=1\quad\text{if $p\gg n^{-1/\delta}$},\]
where $\delta=\max\{m(Q_1),m_r(Q_s,Q_{s+1})\}$.

For the $0$-statement, recall that all complete $r$-graphs on at least $r$ vertices are strictly $r$-balanced. For every $j\in[s+1]$, since in any proper vertex-coloring of $Q_j$ the number of vertices receiving the same color is at most $r-1$, we have $\chi(Q_j)\geq v(Q_j)/(r-1)\geq v(Q_{s+1})/(r-1)>r$. By~\Cref{Ramsey_dense}, $(Q_i,Q_j)$ is Ramsey-dense for all $i<j\in[s+1]$. The $0$-statement then follows from~\Cref{our_0_statement}.\qed

\section{Proof of~\Cref{Ramsey_equivalent}}
\label{proof_equivalent}
\begin{proof}[Proof of~\Cref{Ramsey_equivalent}]
If $(F_1,\dots,F_s)$ and $(Q_1,\dots,Q_t)$ coincide up to reordering and isomorphism, then it holds trivially that $\mathcal{R}(F_1,\dots,F_s)=\mathcal{R}(Q_1,\dots,Q_t)$. It suffices to show the other direction.

Suppose that $\mathcal{R}(F_1,\dots,F_s)=\mathcal{R}(Q_1,\dots,Q_t)$. By~\Cref{structural_theorem}, there exists a partition $A_1\dot\cup\dots\dot\cup A_s$ of $[t]$ such that $F_i\to(Q_j)_{j\in A_i}$ for all $i\in[s]$.

We claim that $A_i\neq\emptyset$ for every $i\in[s]$. Indeed, suppose towards a contradiction that $A_s=\emptyset$. Then since $A_1\dot\cup\dots\dot\cup A_{s-1}$ is a partition of $[t]$ such that $F_i\to(Q_j)_{j\in A_i}$ for every $i\in[s-1]$, any $r$-graph that is Ramsey for $(F_1,\dots,F_{s-1})$ is also Ramsey for $(Q_1,\dots,Q_t)$. On the other hand, let $G$ be a minimal $r$-graph that is Ramsey for $(F_1,\dots,F_s)$ and let $G'$ denote an $r$-graph obtained from $G$ by deleting an arbitrary edge $e$. Due to the minimality of $G$, $G'\not\to(F_1,\dots,F_s)$. Moreover, every $(s-1)$-edge-coloring of $G'$ corresponds to an $s$-edge-coloring of $G$ in which $e$ is the only edge in the $s$-th color. Then since $G\to(F_1,\dots,F_s)$ and $F_s$ contains at least two edges, we have that every $(s-1)$-edge-coloring of $G'$ yields a monochromatic copy of $F_i$ in the $i$-th color for some $i\in[s-1]$. Namely, $G'\to(F_1,\dots,F_{s-1})$ and thus $G'\to(Q_1,\dots,Q_t)$. However, this yields that $G'\to(Q_1,\dots,Q_t)$ but $G'\not\to(F_1,\dots,F_s)$, a contradiction to $\mathcal{R}(F_1,\dots,F_s)=\mathcal{R}(Q_1,\dots,Q_t)$.

Therefore, we have a partition $A_1\dot\cup\dots\dot\cup A_s$ of $[t]$ such that $A_i\neq\emptyset$ and $F_i\to(Q_j)_{j\in A_i}$ for every $i\in[s]$. By symmetry there is also a partition $B_1\dot\cup\dots\dot\cup B_t$ of $[s]$ such that $B_j\neq\emptyset$ and $Q_j\to(F_i)_{i\in B_j}$ for every $j\in[t]$. In particular, this implies that $s=t$ and $|A_1|=\dots=|A_s|=|B_1|=\dots=|B_t|=1$.

To show that $(F_1,\dots,F_s)$ and $(Q_1,\dots,Q_t)$ coincide up to reordering and isomorphism, we shall find a bijection $f:[s]\to[t]$ such that, for any $i\in[s]$, $F_i$ is isomorphic to $Q_{f(i)}$. Assume without loss of generality that $e(Q_1)\leq\dots\leq e(Q_t)$. Let $i\in[s]$ and $j\in[t]$ be such that $A_i=\{t\}$ and $B_j=\{i\}$. Then $F_i\to(Q_t)$ and $Q_j\to(F_i)$, implying that $Q_j$ contains a copy of $Q_t$. Since $Q_j$ contains no isolated vertices and $e(Q_j)\leq e(Q_t)$, $Q_j$ and $Q_t$ are isomorphic, from which it follows that $F_i$ and $Q_t$ are isomorphic. Observe that the tuples $(F_1,\dots,F_{i-1},F_{i+1},F_s)$ and $(Q_1,\dots,Q_{t-1})$ remain Ramsey equivalent. By iterating the argument, we obtain a bijection $f:[s]\to[t]$ such that $F_i$ is isomorphic to $Q_{f(i)}$ for every $i\in[s]$, as desired.
\end{proof}

\section{Proof of~\Cref{structural_theorem}}
\label{proof_structural}
The proof of~\Cref{structural_theorem} relies on the following key lemma, which generalizes a statement of Graham, {\L}uczak, R\"odl, and Ruci\'nski~\cite[Theorem~2.1]{graham2002} for graph cliques. Given an $r$-graph $F$, an edge-partition of $F$ is an ordered tuple of $r$-graphs on the vertex set $V(F)$, such that their edge sets partition $E(F)$.
\begin{lemma}
\label{partition_Ramsey}
Let $r\geq2$ and $s,t\geq1$ be integers. Let $Q_1,\dots,Q_t\in\mathcal{X}_r\cup\mathcal{Y}_r$ be $r$-graphs. Then for any $r$-graph $H$ with an edge-partition $(H_1,\dots,H_t)$, where $H_j$ is $Q_j$-free for each $j\in[t]$, there exists an $r$-graph $G$ with an edge-partition $(G_1,\dots,G_t)$, such that
\begin{itemize}
    \item $G_j$ is $Q_j$-free for each $j\in[t]$;
    \item any $s$-edge-coloring of $G$ yields a copy $\tilde{H}$ of $H$ with the corresponding edge-partition $(\tilde{H}_j)_{j\in[t]}$, such that $\tilde{H_j}$ is monochromatic and is contained in $G_j$ for each $j\in[t]$.
\end{itemize}
\end{lemma}

Since the proof of~\Cref{partition_Ramsey} uses a result of Ne\v{s}et\v{r}il and R\"odl~\cite{nesetril1989}, originally stated in different terminology, we defer its proof to the next section.

\begin{proof}[Proof of~\Cref{structural_theorem}]
If there is a partition $A_1\dot\cup\dots\dot\cup A_s$ of $[t]$ such that $F_i\to(Q_j)_{j\in A_i}$ for all $i\in[s]$, then obviously every $r$-graph $G$ that is Ramsey for $(F_1,\dots,F_s)$ is also Ramsey for $(Q_1,\dots,Q_t)$.

For the other direction, we let $\mathcal{S}$ be the collection of all $s$-partitions of $[t]$. Suppose that for any $\mathcal{A}=(A_1,\dots,A_s)\in\mathcal{S}$ there exists some index $i(\mathcal{A})\in[s]$ such that $F_{i(\mathcal{A})}\not\to(Q_j)_{j\in A_{i(\mathcal{A})}}$. If there are multiple such indices, we simply fix $i(\mathcal{A})$ to be one of them. We shall construct an $r$-graph $G$ that is Ramsey for $(F_1,\dots,F_s)$ but not for $(Q_1,\dots,Q_t)$.

For every $\mathcal{A}=(A_1,\dots,A_s)\in\mathcal{S}$ we let $\tilde{F}_{i(\mathcal{A})}$ be a copy of $F_{i(\mathcal{A})}$, chosen so that $\tilde{F}_{i(\mathcal{A})}$ and $\tilde{F}_{i(\mathcal{A}')}$ are vertex-disjoint for all distinct $\mathcal{A},\mathcal{A}'\in\mathcal{S}$. Since $\tilde{F}_{i(\mathcal{A})}\not\to(Q_j)_{j\in A_{i(\mathcal{A})}}$, there exists an edge-partition $(\tilde{F}^{j}_{i(\mathcal{A})})_{j\in[t]}$ of $\tilde{F}_{i(\mathcal{A})}$, such that $\tilde{F}^{j}_{i(\mathcal{A})}$ is $Q_j$-free for $j\in A_{i(\mathcal{A})}$ and is empty for $j\not\in A_{i(\mathcal{A})}$. In particular, \[\tilde{F}_{i(\mathcal{A})}=\bigcup_{j\in A_{i(\mathcal{A})}}\tilde{F}^{j}_{i(\mathcal{A})}.\]

For each $j\in[t]$, we let $H_j=\bigcup_{\mathcal{A}\in\mathcal{S}}\tilde{F}^{j}_{i(\mathcal{A})}$. Since $H_j$ is a vertex-disjoint union of $Q_j$-free $r$-graphs, $H_j$ is also $Q_j$-free. Let $H$ be the union of $H_1,\dots,H_t$. Observe that $(H_j)_{j\in[t]}$ is an edge-partition of $H$. Then by~\Cref{partition_Ramsey} there exists an $r$-graph $G$ with an edge-partition $(G_j)_{j\in[t]}$, such that
\begin{enumerate}[(a)]
    \item\label{(a)} $G_j$ is $Q_j$-free for each $j\in[t]$;
    \item\label{(b)} any $s$-edge-coloring of $G$ yields a copy $\tilde{H}$ of $H$ with the corresponding edge-partition $(\tilde{H}_j)_{j\in[t]}$, such that $\tilde{H_j}$ is monochromatic and is contained in $G_j$ for each $j\in[t]$. 
\end{enumerate}

From item~\eqref{(a)} it follows immediately that $G\not\to(Q_1,\dots,Q_t)$.

To show that $G\to(F_1,\dots,F_s)$, we consider an arbitrary $s$-edge-coloring of $G$. By item~\eqref{(b)}, for each $j\in[t]$, there exists a monochromatic copy $\tilde{H}_j$ of $H_j$ in $G_j$. Moreover, $\bigcup_{j\in[t]}\tilde{H}_j$ is a copy of $H$. Consider the $s$-partition $\mathcal{A}=(A_1,\dots,A_s)\in\mathcal{S}$, where
\[A_i=\{j\in[t]:\text{$\tilde{H}_j$ is in the $i$-th color}\}\]
for each $i\in[s]$. Then there exists a copy of $\bigcup_{j\in A_{i(\mathcal{A})}}\tilde{F}^{j}_{i(\mathcal{A})}$ in $\bigcup_{j\in A_{i(\mathcal{A})}}\tilde{H}_j$, which is in the $i(\mathcal{A})$-th color. Since $\bigcup_{j\in A_{i(\mathcal{A})}}\tilde{F}^{j}_{i(\mathcal{A})}=\tilde{F}_{i(\mathcal{A})}$ is a copy of $F_{i(\mathcal{A})}$, we conclude that $G\to(F_1,\dots,F_s)$.
\end{proof}

\section{Proof of~\Cref{partition_Ramsey}}
\label{proof_lemma}
In this section we establish~\Cref{partition_Ramsey} from a classical result of Ne\v{s}et\v{r}il and R\"odl~\cite{nesetril1989} on Ramsey theory for systems. To state this result, we need several definitions concerning systems. The original definitions and statements in~\cite{nesetril1989} are formulated also for the ordered setting. As this is not needed for our application, we omit that aspect.

Let $t\geq1$ be an integer and $\gamma=(r_1,\dots,r_t)$ be a tuple of positive integers. A $\gamma$-system is a pair $(V,\mathcal{E})$, where $V$ is a finite set and $\mathcal{E}=(E_1,\dots,E_t)$ is a tuple of pairwise disjoint sets with $E_j\subseteq\binom{V}{r_j}$ for each $j\in[t]$. The elements in $V$ are called the vertices of the system, while the elements in $E_j$ are called edges.

For two $\gamma$-systems $A=(V(A),\mathcal{E}(A))$ and $B=(V(B),\mathcal{E}(B))$, we say that $A$ is a subsystem of $B$ if $V(A)\subseteq V(B)$ and $E_j(A)=\binom{V(A)}{r_j}\cap E_j(B)$ for every $j\in[t]$. Two $\gamma$-systems $A$ and $\tilde{A}$ are called isomorphic if there exists a bijection $f:V(A)\to V(\tilde{A})$ such that, for any $W\subseteq V(A)$ and $j\in[t]$, $W\in E_j(A)$ if and only if $f(W)\in E_j(\tilde{A})$. We say that a $\gamma$-system $B$ contains a copy of a $\gamma$-system $A$ if $B$ has a subsystem isomorphic to $A$. Let $\binom{B}{A}$ denote the collection of all copies of $A$ in $B$.

A system is called connected if, for any two vertices $x$ and $y$, there exists a sequence of edges such that consecutive edges intersect, with $x$ contained in the first edge and $y$ contained in the last edge. A subsystem $\tilde{B}$ of a system $B$ is called a cut if the system obtained by removing the vertices of $\tilde{B}$ from $B$ is disconnected. Given two systems $A$ and $B$, we say that $B$ is $\mathrm{Hom}(A)$-connected if $B$ contains no cut $\tilde{B}$ that is homomorphic to $A$. Here, $\tilde{B}$ is said to be homomorphic to $A$ if there is a mapping $f:V(\tilde{B})\to v(A)$ such that, for any $W\subseteq V(\tilde{B})$, if $W\in E_j(\tilde{B})$ then $f(W)\in E_j(A)$.

A system is called irreducible if for every two vertices there is an edge containing them. 

Let $s\geq1$ be an integer and $A,B,C$ be $\gamma$-systems. We denote $C\to(B)^A_s$ if for any $s$-coloring of $\binom{C}{A}$, there exists a copy $\tilde{B}$ of $B$ in $C$ such that all members in $\binom{\tilde{B}}{A}$ receive the same color.

\begin{theorem}[Ne\v{s}et\v{r}il--R\"odl~\cite{nesetril1989}]
\label{system_Ramsey}
Let $s,t\in\mathbb{N}$ and $\gamma=(r_1,\dots,r_t)$ be a tuple of positive integers. For any $\gamma$-systems $A$ and $B$, there exists a $\gamma$-system $C$ such that
\[C\to(B)^A_s.\]
Moreover, if $B$ is $\mathcal{F}$-free, where $\mathcal{F}$ is a family of $\mathrm{Hom}(A)$-connected systems and irreducible systems, then $C$ can also be chosen $\mathcal{F}$-free.
\end{theorem}

\begin{proof}[Proof of~\Cref{partition_Ramsey}]
Let $r\geq2$ and $s,t\geq1$ be integers. Let $Q_1,\dots,Q_t\in\mathcal{X}_r\cup\mathcal{Y}_r$ be $r$-graphs. Let $H$ be an $r$-graph with an edge-partition $(H_1,\dots,H_t)$, where $H_j$ is $Q_j$-free for each $j\in[t]$.

Let $\gamma=(\underbrace{r,\dots,r}_{\text{$t$ times}})$. We define the $\gamma$-system $B=(V(B),\mathcal{E}(B))$, where
\[V(B)=V(H) \quad\text{and}\quad \mathcal{E}(B)=(E(H_1),\dots,E(H_t)).\]

Further, for each $j\in[t]$, we define the $\gamma$-system $A_j=(V(A_j),\mathcal{E}(A_j))$ by
\[V(A_j)=[r]\quad\text{and}\quad E_i(A_j)=\begin{cases}
\{[r]\} & \text{if $i=j$},\\
\emptyset & \text{otherwise},
\end{cases}\]
and the $\gamma$-system $F_j=(V(F_j),\mathcal{E}(F_j))$ by
\[V(F_j)=V(Q_j) \quad\text{and}\quad E_i(F_j)=\begin{cases}
E(Q_j) & \text{if $i=j$},\\
\emptyset & \text{otherwise}.
\end{cases}\]

Since $Q_j\in\mathcal{X}_r\cup\mathcal{Y}_r$, the $\gamma$-system $F_j$ is either $\mathrm{Hom}(A_i)$-connected for every $i\in[t]$ or irreducible. Let $\mathcal{F}=\{F_1,\dots,F_t\}$. As $H_j$ is $Q_j$-free for each $j\in[t]$, we have that $B$ is $\mathcal{F}$-free.

Applying~\Cref{system_Ramsey} iteratively, we obtain a sequence of $\mathcal{F}$-free $\gamma$-systems $C_1,\dots,C_t$, such that
\[C_1\to(B)^{A_1}_s \quad\text{and}\quad C_{i}\to(C_{i-1})^{A_i}_s\text{ for $i=2,\dots,t$}.\] 
Let $G$ be the $r$-graph with the edge-partition $(G_1,\dots,G_t)$, where
\[V(G)=V(C_t) \quad\text{and}\quad E(G_j)=E_j(C_t)\]
for every $j\in[t]$. Since $C_t$ is $\mathcal{F}$-free, each $G_j$ is $Q_j$-free. Moreover, for any $s$-coloring of $\bigcup_{j\in[t]}\binom{C_t}{A_j}$, there exists a copy $\tilde{B}$ of $B$ in $C_t$ such that, for each $j\in[t]$, $\binom{\tilde{B}}{A_j}$ is monochromatic. This guarantees that any $s$-edge-coloring of $G$ yields a copy $\tilde{H}$ of $H$ with the corresponding edge-partition $(\tilde{H}_j)_{j\in[t]}$, such that $\tilde{H_j}$ is monochromatic and is contained in $G_j$ for each $j\in[t]$.
\end{proof}

\section{Concluding remarks}
\label{concluding}
In this work we investigate the Ramsey classes of random hypergraphs. More precisely, we study the threshold probability for $\mathcal{R}(H;s)\subseteq\mathcal{R}(Q_1,\dots,Q_t)$, where $H=H^{(r)}(n,p)$ is a random $r$-graph and $Q_1,\dots,Q_t$ are fixed $r$-graphs satisfying $m_r(Q_1)\geq\dots\geq m_r(Q_t)>0$ and $m_r(Q_{s+1})>1$. This problem reveals a natural connection between random and structural Ramsey theory. In the genuinely interesting regime $s<t$, we give a $1$-statement for all $r$-graphs (see~\Cref{our_1_statement}), and a corresponding $0$-statement for a large class of $Q_1,\dots,Q_t$, namely those chosen from $\mathcal{X}_r\cup\mathcal{Y}_r$ with certain pairs being Ramsey-dense (see~\Cref{our_0_statement}). The threshold appears to be governed by two density parameters, $\mu=\max\{m(Q_1),\dots,m(Q_t)\}$ and $\sigma=\min\{m_r(Q_i,Q_{s+1}):i\in[s]\}$. We note that when $\mu<\sigma$, the thresholds in~\Cref{our_1_statement} and~\Cref{our_0_statement} can be sharpened to $p\geq Cn^{-1/\sigma}$ and $p\leq Dn^{-1/\sigma}$ for some constants $C,D>0$, since the coarse threshold arises solely from the application of~\eqref{small_subgraphs}.

There are two obvious restrictions in~\Cref{our_0_statement}. First, we require $(Q_i,Q_j)$ to be Ramsey-dense for all $i<j\in[s+1]$. As discussed in~\Cref{introduction}, this condition is, to some extent, necessary, since otherwise the threshold for $H^{(r)}(n,p)\not\to(Q_i,Q_j)$ need not be of the form $n^{-1/m_r(Q_i,Q_j)}$. However, we expect that a more careful analysis could weaken this condition, reducing the number of pairs that need to be Ramsey-dense. On the other hand, it was established in~\cite{christoph2025} that all graph pairs are Ramsey-dense, and several sufficient conditions for a pair of r-graphs to be Ramsey-dense are given in~\cite{bowtell2023}. One of such conditions (see~\Cref{Ramsey_dense}) asserts that $(Q_i,Q_j)$ is Ramsey-dense whenever $Q_j$ contains a strictly $r$-balanced subhypergraph with chromatic number larger than $r$. For $r$-graphs chosen from $\mathcal{X}_r\cup\mathcal{Y}_r$, this is not a particularly restrictive property, since the members of $\mathcal{X}_r\cup\mathcal{Y}_r$ are highly connected and often dense.

The other restriction in~\Cref{our_0_statement} is that $Q_1,\dots,Q_t$ are chosen within $\mathcal{X}_r\cup\mathcal{Y}_r$. This restriction stems from~\Cref{structural_theorem} and ultimately from the result of Ne\v{s}et\v{r}il and R\"odl~\cite{nesetril1989} on Ramsey theory for systems (see~\Cref{system_Ramsey}). In their proof of~\Cref{system_Ramsey}, systems with the desired properties are constructed via iterative amalgamations. Such operations generate many sparse substructures and therefore can avoid only highly connected configurations, such as those described by $\mathcal{X}_r\cup\mathcal{Y}_r$. Consequently, a more general version of~\Cref{structural_theorem} would require methods different from those of Ne\v{s}et\v{r}il and R\"odl~\cite{nesetril1989}. On the other hand,~\Cref{structural_theorem} states that, for certain choices of $Q_1,\dots,Q_t$, the trivial sufficient condition for $\mathcal{R}(F_1,\dots,F_s)\subseteq\mathcal{R}(Q_1,\dots,Q_t)$ is also necessary. Intuitively, this indicates that being Ramsey for $(Q_1,\dots,Q_t)$ is rather restrictive, and thus naturally forces each $Q_i$ to be dense or highly connected. In this light, it is reasonable to suspect that~\Cref{structural_theorem} might fail when $Q_1,\dots,Q_t$ are sparse or less connected, which in turn suggests that the threshold for $\mathcal{R}(H^{(r)}(n,p);s)\subseteq\mathcal{R}(Q_1,\dots,Q_t)$ could exhibit different behavior.

Finally, another direct application of~\Cref{structural_theorem} is to establish the existence of $r$-graphs with counterintuitive Ramsey properties. We illustrate several examples involving complete $r$-graphs.

\begin{corollary}
\label{asymmetric_Ramsey}
Let $k>\ell\geq r\geq2$ be integers. For each of the following properties, there exists an $r$-graph $G$ satisfying it:
\begin{enumerate}[(i)]
    \item\label{(i)} $G\to(K^{(r)}_q,K^{(r)}_\ell)$ and $G\not\to(K^{(r)}_k,K^{(r)}_k)$, where $q=R(K^{(r)}_k,K^{(r)}_k)-1$;
    \item\label{(ii)} $G\to(K^{(r)}_{k-1},K^{(r)}_{k-1},K^{(r)}_\ell)$ and $G\not\to(K^{(r)}_k,K^{(r)}_\ell)$;
    \item\label{(iii)} $G\to(K^{(r)}_{k+1},K^{(r)}_{k-1},K^{(r)}_\ell)$ and $G\not\to(K^{(r)}_k,K^{(r)}_k,K^{(r)}_\ell)$.
\end{enumerate}
\end{corollary}
\begin{proof}[Proof of~\Cref{asymmetric_Ramsey}]
As items~\eqref{(i)},~\eqref{(ii)}, and~\eqref{(iii)} all follow directly from~\Cref{structural_theorem} upon verifying the required conditions, we prove only item~\eqref{(i)} here for simplicity.

Let $F_1=K^{(r)}_q$, $F_2=K^{(r)}_\ell$, and $Q_1=Q_2=K^{(r)}_k$, where $q=R(K^{(r)}_k,K^{(r)}_k)-1$. Fix an arbitrary partition $A_1\dot\cup A_2=[2]$. Since $Q_1,Q_2\in\mathcal{Y}_r$, to conclude that $\mathcal{R}(F_1,F_2)\not\subseteq\mathcal{R}(Q_1,Q_2)$, by~\Cref{structural_theorem} it suffices to show that there is some $i\in[2]$ such that $F_i\not\to(Q_j)_{j\in A_i}$. If $A_2\neq\emptyset$, then from $k>\ell$ it follows that $F_2\not\to(Q_j)_{j\in A_2}$. If $A_2=\emptyset$, then $A_1=[2]$. Since $q<R(K^{(r)}_k,K^{(r)}_k)=R(Q_1,Q_2)$, we have that $F_1\not\to(Q_j)_{j\in A_1}$.
\end{proof}

We remark that there are alternative proofs of item~\eqref{(i)} and of most cases of item~\eqref{(ii)} without recourse to~\Cref{structural_theorem}. Indeed, item~\eqref{(i)} was recently proved by Sviridenkov~\cite{sviridenkov2026} by considering the $r$-uniform shadow of a random hypergraph of higher uniformity, while item~\eqref{(ii)} under an additional assumption $\ell>r^2-r$ can be obtained by utilizing multiple results on Ramsey properties of random hypergraphs from~\cite{bowtell2023,conlon2016,friedgut2010}. For item~\eqref{(iii)}, however, no alternative proof is known to us.

\vspace{1em}
\textbf{Acknowledgement.} The author is grateful to Maria Axenovich for many fruitful discussions and for bringing the reference~\cite{graham2002} to his attention.

\end{document}